\documentclass[
a4paper, 
reqno, 
oneside, 
12pt
]{amsart}

\usepackage[utf8]{inputenc}

\usepackage[dvipsnames]{xcolor} 
\colorlet{cite}{red}
\definecolor{pgreen}{rgb}{0,0.4,0.2}
\usepackage{tikz}  

\usetikzlibrary{arrows,positioning}
\tikzset{ 
	baseline=-2.3pt,
	text height=1.5ex, text depth=0.25ex,
	>=stealth,
	node distance=2cm,
	mid/.style={fill=white,inner sep=2.5pt},
}

\usepackage{lmodern}
\usepackage{tikz-cd}
\usepackage{amsthm, amssymb, amsfonts}

\usepackage{amsthm, amssymb, amsfonts}	
\usepackage[all]{xy}
\usepackage{graphicx,caption,subcaption}
\usepackage{braket}  
\usepackage{microtype}
\usepackage{DotArrow}
\usepackage[makeroom]{cancel}
\usepackage[%
bookmarks=true,			
unicode=true,			
pdftitle={LiftingProblems},		%
pdfauthor={Herrera, deMelo, Valencia},	%
pdfkeywords={Lifting homotopies, orbit spaces},	
colorlinks=true,		
linkcolor=black,			
citecolor=black,		
filecolor=magenta,		
urlcolor=RoyalBlue			
]{hyperref}				
\usepackage{cleveref}

\newtheoremstyle{mydef}
{}		
{}		
{}		
{}		
{\scshape}	
{. }		
{ }		
{\thmname{#1}\thmnumber{ #2}\thmnote{ #3}}	

\newtheorem{theorem}{Theorem}[section]
\newtheorem*{theorem*}{Theorem}
\newtheorem{proposition}[theorem]{Proposition}
\newtheorem*{proposition*}{Proposition}
\newtheorem{lemma}[theorem]{Lemma}
\newtheorem*{lemma*}{Lemma}
\newtheorem{corollary}[theorem]{Corollary}
\newtheorem*{corollary*}{Corollary}
\theoremstyle{definition}

\theoremstyle{remark}
\newtheorem{remark}[theorem]{Remark}

\newtheorem*{conjecture*}{Conjecture}
\usepackage{tikz}

\def\to{\rightarrow}
\def\toto{\rightrightarrows}

\def\act{\curvearrowright}

\newcommand{\rr}{\rightrightarrows}

\author{M. de Melo, J. S. Herrera-Carmona, F. Valencia}
\address{}
\date{\today}

\address{Mateus de Melo - Departamento de Matemática, Universidade Federal do Espírito Santo, Avenida Fernando Ferrari 514, Vitória, 29075-910, Espírito Santo - Brazil.
	\vspace{0.5cm}
	\newline  
	Juan Sebastian Herrera-Carmona - Departamento de Matemática, Universidade Federal do Paraná, Setor de Ciências Exatas - Centro Politécnico, Curitiba, 81531-980 Paraná - Brazil. 
	\vspace{0.5cm}
	\newline  
	Fabricio Valencia - Instituto de Matem\'atica e Estat\'istica, Universidade de S\~ao Paulo, Rua do Mat\~ao 1010, Cidade Universit\'aria, S\~ao Paulo, 05508-090 S\~ao Paulo - Brazil.
	\vspace{0.5cm}
	\newline  
	\phantom{xx}
	mateus.melo@ufes.br, sebastianherrera@ufpr.br, fabricio.valencia@ime.usp.br}

\title{Stratified vector fields on orbit spaces}

\begin{document}
	\maketitle

	\begin{abstract}
Using Morita type stratifications, we establish a one-to-one correspondence between geometric vector fields on a separated differentiable stack and stratified vector fields on its orbit space. This correspondence enables us to derive a stacky version of the generalized Gauss lemma and to prove a smooth version of Palais' covering isotopy theorem for a class of proper Lie groupoids, thereby extending the classical result for proper Lie group actions.
	\end{abstract}
	
	\section{Introduction}

A stratified space can be roughly understood as a well-behaved topological space equipped with a partition into smooth manifolds \cite{Pflaum0}. This concept is particularly useful in the study of certain types of singular spaces, such as those arising as quotient spaces of well-behaved singular foliations, as for instance, orbit spaces of proper Lie group actions \cite{Davis,Kankaanrinta,GSchwarz}, or, more generally, orbit spaces of proper Lie groupoids \cite{CM}. In fact, there exist several stratifications of the base of a proper Lie groupoid that induce stratifications on its orbit space (see, e.g., \cite{CM,PosthumaTangWang}). However, the stratification by Morita types studied in \cite{CM} is canonical in the sense that it depends only on the transversal information associated with the Lie groupoid. On the one hand, the stratification by Morita types is motivated by the notion of normal orbit type stratification, which was initially developed by Schwarz in \cite{GSchwarz} (see also \cite{Davis}) for actions of compact Lie groups, and later explored by Kankaanrinta in the case of proper actions of non-compact Lie groups \cite{Kankaanrinta}. On the other hand, since the stratification by Morita types depends only on the transversal information of a Lie groupoid, it can be used to extend various topological techniques to the stacky setting. This has been recently illustrated by Ortiz and Valencia in the context of Morse theory for Lie groupoids and their differentiable stacks \cite{Valencia}, where a connection is established with the classical theory of stratified Morse theory developed by Goresky and MacPherson in \cite{GM}.

It is well known that stratified vector fields on the orbit space of a proper Lie group action on a manifold can be described in terms of invariant vector fields on the manifold \cite{Davis,Kankaanrinta,GSchwarz}. As a notable consequence, a smooth version of Palais’ covering homotopy theorem (see \cite{Palais1,Palais2}) holds for proper actions of Lie groups. These results naturally rely on the well-behaved properties of the normal orbit type stratification that can be defined in this context.

The aim of this paper is to extend the above result to the orbit space of a proper Lie groupoid, taking into account the canonical stratification by Morita types. To be precise, we first introduce some terminology. Let $G\rr M$ be a proper Lie groupoid and let $X=M/G$ denote its orbit space. We say that points in $M$ are Morita type equivalent if their normal isotropy representations are isomorphic. 
Passing to the connected components of the corresponding partition defines a canonical Whitney stratification of $M$, which in turn induces a canonical Whitney stratification of the orbit space $X$. 
Therefore, we can view $X$ as a smooth stratified space whose strata are completely determined by the Morita types in $M$. If $C^\infty(X)$ denotes the algebra of smooth functions on $X$ then we define the Lie algebra of stratified vector fields $\mathfrak{X}^\infty(X)$ as the subalgebra of $\mathfrak{Der}(C^{\infty}(X))$ consisting of derivations that preserve the ideals of $C^{\infty}(X)$ vanishing along the strata. A multiplicative vector field on $G$ is a vector field whose flow consists of local Lie groupoid automorphisms of $G$. 
The space of all multiplicative vector fields on $G$ forms a Lie algebra, which we denote by $\mathfrak{X}_{m}(G)$. 
More importantly, there exists a crossed module of Lie algebras $(\mathfrak{X}_m(G),\Gamma(A),\delta,D)$, where $A$ is the Lie algebroid of $G$, which allows us to define geometric vector fields on the separated differentiable stack $[M/G]$ presented by $G\rr M$ as elements of the quotient Lie algebra $\mathfrak{X}([M/G]):=\mathfrak{X}_{m}(G)/\textnormal{im}(\delta)$.

In these terms, the first main result of this paper can be stated as follows.

\begin{theorem*}[A]
The canonical projection $\pi:M\to X$ induces a Lie algebra homomorphism $\pi_*:\mathfrak{X}_{m}(G) \to \mathfrak{Der}(C^{\infty}(X))$ such that its image is the Lie algebra of preserving vector fields $\mathfrak{X}^{\infty}(X)$. Moreover, there is an exact sequence of Lie algebras:
\begin{center}
	\begin{tikzcd}
		0\arrow[r]&\mathrm{im}(\delta)\arrow[r, hook]&\mathfrak{X}_{m}(G)\arrow[r,"\pi_*"]&\mathfrak{X}^{\infty}(X)\arrow[r]&0.
	\end{tikzcd}
\end{center}

\end{theorem*}

An immediate and significant consequence of this result is that geometric vector fields on the separated differentiable stack $[M/G]$ induce stratified vector fields on the orbit space $X$. In fact, there is a Lie algebra isomorphism $\mathfrak{X}([M/G])\simeq \mathfrak{X}^{\infty}(X)$. A second important consequence concerns two notable applications: the first is a stacky version of the Gauss lemma, relevant in geometric analysis. More precisely, if $d_N$ denotes the geodesic distance on $X$ induced by a groupoid Riemannian metric on $G$ then we obtain the following:

 \begin{proposition*}[B]
	
	Let $G \rightrightarrows M$ be a Riemannian groupoid. Then, for any $x \in M$ such that the orbit $\mathcal{O}_x$ is an embedded submanifold, there exists $\epsilon > 0$ verifying that:
	\begin{enumerate}
		\item the distance function $r_x \colon B^{\circ}_{d_N}([x],\epsilon) \to \mathbb{R}$, defined by $r_x([y]) = d_N([x],[y])$, is smooth on the punctured neighborhood $B^{\circ}_{d_N}([x],\epsilon)$, and
		\item its gradient $\nabla r_x:=\pi_*(\nabla r_{\mathcal{O}_x})$ is a stratified vector field on $B^{\circ}_{d_N}([x],\epsilon)$.
	\end{enumerate}
	
\end{proposition*}

In particular, as every proper Lie groupoid admits a Riemannian groupoid metric \cite[Thm. 4.13]{dHF}, the previous result applies to any Lie groupoid of this type. The second consequence of Theorem (A) is a smooth version of Palais' covering isotopy theorem for a class of proper Lie groupoids. Namely, let $H \rightrightarrows N$ be another proper Lie groupoid with orbit space $Y$. Hence, we get that:
 
\begin{theorem*}[C]
	Let $\psi: (G \rightrightarrows M) \to (H \rightrightarrows N)$ be a Lie groupoid morphism and let $H$ be a Lie groupoid such that the canonical projection $\pi : N \to Y$ is proper. Let $F : \mathbb{R} \times X \to Y$ be a stratified map such that the map $F_\tau:=F|_{\{\tau\}\times X}$ is a normally transverse embedding for every $\tau\in \mathbb{R}$ and $F_0 = \overline{\psi}:X\to Y$. Then, the following assertions hold true:
	\begin{enumerate}
		\item There exists a Lie groupoid morphism $\Psi : \mathbb{R} \times G \to H$ such that $\Psi_0 = \psi$ and the induced map between orbit spaces $\overline{\Psi}$ agrees with $F$.
		\item If $\psi$ is a Morita fibration then $\Psi_\tau$ is also a Morita fibration for all $\tau\in \mathbb{R}$. 
	\end{enumerate}
Moreover, any two lifts for $F$ differ by a smooth natural transformation.
\end{theorem*}

Roughly speaking, the homotopy lifting property is a common feature of locally trivial fibrations. In fact, Ehresmann’s theorem states that proper submersions between manifolds are locally trivial fibrations, and Thom's first isotopy lemma \cite[Prop.~11.1]{Mather} ensures that stratified proper submersions are locally trivial. Consequently, Theorem (C) suggests a connection between the liftability of a stratified isotopy of an orbit space and the property of being a stacky locally trivial fibration (compare  \cite{dHF2}). 

It is worth mentioning that further applications are expected from these results, such as a stacky version of the Myers–Steenrod theorem for separated Riemannian stacks \cite{dMHCV}.

The paper is organized as follows. In Section \ref{S:2}, we define Morita-type stratifications of proper Lie groupoids, following closely the conventions and terminology introduced in \cite{CM}. In particular, we describe the fibers of what can be regarded as the tangent bundle of a stratified orbit space, showing that they have the structure of a cone (see Proposition \ref{Pro1}). Section \ref{S:3} is devoted to developing the necessary machinery to establish the natural identification provided by our first main result, which arises from describing the relationship between multiplicative vector fields on proper Lie groupoids and stratified vector fields on the corresponding orbit spaces. This is the content of Theorem \ref{thm1} and Proposition \ref{prop:exact-sequence}. The first application of this identification, namely the stacky version of the generalized Gauss lemma, is exhibited in Proposition \ref{Pro2}. In Section \ref{S:4}, we establish a smooth version of Palais’ covering isotopy theorem for a class of proper Lie groupoids, thereby proving our second main result (see Theorem \ref{thm2} and Proposition \ref{cor1}). To this end, we begin by introducing the notions of Zariski normally transverse and stacky normally transverse maps between smooth stratified orbit spaces, and conclude by describing some of their local properties.

	\vspace{.5cm}
	{\bf Acknowledgments:} We would like to thank Marcos Alexandrino, Olivier Brahic, and Ivan Struchiner for several valuable discussions related to this work. M. de Melo was supported by Grant 01/2024 (FAP–UFES). F. Valencia was supported by Grant 2024/14883-6 from the São Paulo Research Foundation (FAPESP). M. de Melo also thanks the Geometry Group at IME-USP, especially M. Alexandrino and I. Struchiner, for their support during visits to the institute.
	
\section{Morita type stratification}\label{S:2}

This section is devoted to introducing as well as describing some properties of the type of stratification for the orbit space of a proper Lie groupoid that we shall be dealing with throughout the paper. Let $X$ be a Hausdorff second-countable paracompact topological space. A \textbf{stratification} of $X$ is a locally finite partition $\mathcal{S}=\{X_j: j\in I\}$ of $X$ such that its members satisfy:
\begin{enumerate}
	\item each $X_j$, endowed with the subspace topology, is a locally closed, connected subspace of $X$, carrying a given structure of a smooth manifold, and
	\item \textbf{frontier condition}: the closure of each $X_j$ is the union of $X_j$ with members of $\mathcal{S}$ of strictly lower dimension.
\end{enumerate}

\noindent The members $X_j\in \mathcal{S}$ are called the \textbf{strata} of the stratification. In the presence of smooth structures, there
are several other conditions which are often imposed, with the
Whitney’s conditions being the most common of them. Namely, let $M$ be a smooth manifold and let $\mathcal{S}$ denote a stratification on $M$. We say that a pair of strata $(S,R)$ such that $S$ is in the closure of $R$ satisfies the \textbf{Whitney condition (A)} if the following holds:

\begin{enumerate}
	\item[(A)] Let $(x_n)$ be any sequence of points in $R$ such that $x_n$ converges to a point $x\in S$ and $T_{x_n}R$ converges to $\tau\subset T_x M$ in the Grassmannian of subspaces of $TM$ of dimension $\dim R$. Then $T_xS\subset \tau$.
\end{enumerate}

\noindent Let $\phi:U\to \mathbb{R}^n$ be a local chart of $M$ around $x\in S$. We say that $(S,R)$ satisfies the \textbf{Whitney condition (B)} at $x$ with respect to the chart $(U,\phi)$ if verifies the following:

\begin{enumerate}
	\item[(B)] Let $(x_n)$ be a sequence as in (A), with $x_n\in U \cap R$ and let $(y_n)$ denote a sequence of points of $U\cap S$, converging to $x$, such that the sequence of lines $\overline{\phi(x_n)\phi(y_n)}$
	converges in the projective $n$-space to a line $l$. Then $(d\phi_x)^{-1}(l)\subset  \tau$.
\end{enumerate}

\noindent The pair of strata $(S,R)$ is said to satisfy the Whitney condition (B) if the above condition holds for any point $x\in S$ and any chart around $x$. Although we have used charts in the definition of Whitney condition (B), the
condition is actually independent of the chart chosen. A stratification is called a \textbf{Whitney stratification} if all pairs of strata satisfy the Whitney conditions (A) and (B).

From now on, we assume that the reader is familiar with the notion of a Lie groupoid and its fundamental underlying geometric and topological aspects \cite{dH,MM}. Let $G\rightrightarrows M$ be a proper Lie groupoid with structural maps $(s,t,m,u,i)$. It is well-known that the properness assumption on $G$ implies that the groupoid orbits $\mathcal{O}_x$ are embedded in $M$, the isotropy Lie groups $G_x$ are compact, and the orbit space $M/G$ is Hausdorff, second-countable, and paracompact. Furthermore, the orbit projection $\pi:M\to M/G$ is an open map. Additionally, it turns out that proper Lie groupoids are linearizable, in the sense that around any orbit $\mathcal{O}_x$ there are full groupoid neighborhoods $(G_U\rr U)\subseteq (G\rr M)$ of $G_{\mathcal{O}_x}\rr \mathcal{O}_x$ and $(\nu(G_{\mathcal{O}_x})_V\rr V)\subseteq (\nu(G_{\mathcal{O}_x})\rr \nu(\mathcal{O}_x))$ of $G_{\mathcal{O}_x}\rr \mathcal{O}_x$, seen as the zero section, and a Lie groupoid isomorphism
$$\phi:(\nu(G_{\mathcal{O}_x})_V\rr V)\to (G_U\rr U),$$
which is the identity on $G_{\mathcal{O}_x}\rr \mathcal{O}_x$. Such a linearization $\phi$ is said to be \textbf{invariant} if both $U$ and $V$ above are saturated. It is worth emphasizing that the so-called linearization theorem generalizes several classical and fundamental results in differential geometry, such as Ehresmann’s theorem for submersions, Reeb’s local stability theorem for foliations, and the tube theorem for proper Lie group actions \cite{CS,dHdM2,dHF,Me2,We2,Zung}.

\begin{remark}
A slice at $x\in M$ is an embedded submanifold $S \subseteq M$ of dimension complementary to $\mathcal{O}_x$ such that it is transverse to all orbits it meets and $S\cap \mathcal{O}_x=\left\lbrace x\right\rbrace$. It is well known that for each $x \in M$ we have that $\nu(\mathcal{O}_x) \simeq P_x \times_{G_x} \nu_x(\mathcal{O}_x)$, where $P_x:=s^{-1}(x)$ is a principal
$G_x$-bundle over $\mathcal{O}_x$ along the target map. Therefore, the linearization theorem implies that the $G_x$-invariant subset $W_x = \nu_x(\mathcal{O}_x) \cap V$ gives rise to a slice at $x$, and, additionally, if $V \subset \nu(\mathcal{O}_x)$ is saturated, then one can identity $P_x \times_{G_x} W_x\simeq \nu(G_{\mathcal{O}_x})_V \simeq G_{\mathcal{O}_x} \ltimes V$.

\end{remark}

Let $X$ denote the orbit space of $G$. The Morita type stratifications of $M$ and $X$ introduced below generalize the canonical stratifications induced by proper Lie group actions \cite{CM}. We say that two point $x,y\in M$ are \textbf{Morita type equivalent} if and only if the normal representations $G_x\curvearrowright \nu_x(\mathcal{O}_x)$ and $G_y\curvearrowright \nu_y(\mathcal{O}_y)$ are isomorphic. The partition by Morita types is denoted by $\mathcal{P}_\mathcal{M}(M)$, and we refer to each member of it as a \textbf{Morita type}. If $\alpha$ denotes the isomorphism class associated with the normal representation $G_x\curvearrowright \nu_x(\mathcal{O}_x)$ then the element in $\mathcal{P}_\mathcal{M}(M)$ corresponding to $\alpha$ will be denoted by $M_{(\alpha)}=\lbrace x\in M: [G_x, \nu_x(\mathcal{O}_x)]=\alpha\rbrace$. We also denote by $M_{(x)}$ the Morita type of a point $x\in M$. Points in the same orbit belong to the same Morita type and hence we also obtain a partition by Morita types on the orbit space, which is denoted by $\mathcal{P}_\mathcal{M}(X)$. The orbit projection map $\pi: M\to X$ takes Morita types $M_{(\alpha)}$ in $M$ into Morita types $X_{(\alpha)}=\pi(M_{(\alpha)})$ in $X$.

Following \cite[Sec.~4.6]{CM}, the canonical Whitney stratification by Morita types on $M$, denoted by $\mathcal{S}_{G}(M)$, is the partition on $M$ obtained by passing to connected components of $\mathcal{P}_\mathcal{M}(M)$. Accordingly, the
Whitney stratification by Morita types on the orbit space $X$, denoted by $\mathcal{S}(X)$, is the partition on $X$ obtained
by passing to connected components of $\mathcal{P}_\mathcal{M}(X)$. One finds that if $G$ has only one Morita type then the orbit space $X$ is a smooth manifold and the canonical projection $\pi : M \to X$ becomes a submersion whose fibers are the orbits. More importantly, let $G_{(\alpha)}\rr M_{(\alpha)}$ denote the groupoid $G_{(\alpha)}=s^{-1}(M_{(\alpha)})$ over a fixed Morita type $M_{(\alpha)}$. It follows that $G_{(\alpha)}\rr M_{(\alpha)}$ is a Lie groupoid, $X_{(\alpha)}$ is a smooth manifold and the canonical projection $\pi : M_{(\alpha)}  \to X_{(\alpha)}$ is a submersion. If either the orbit space $X$ is compact or has a finite number of Morita types then it can be embedded into an affine space $\mathbb{R}^d$, compare \cite{CM,FaSea}. Additionally, any Morita equivalence between two proper Lie groupoids induces an isomorphism of differentiable stratified spaces between their orbit spaces \cite{CM}.

We provide now a description of the local model around a stratum $X_{(\alpha)}$ in $\mathcal{P}_\mathcal{M}(X)$. Let $\nu(G_{(\alpha)})\rr \nu(M_{(\alpha)})$ denotes the normal groupoid of $G_{(\alpha)}\rr M_{(\alpha)}$, which can be naturally regarded as a $VB$-groupoid. It turns out that the projection of $\nu(G_{(\alpha)})$ onto $G_{(\alpha)}$ gives rise to a map of the stratified spaces $\bar{\pi} : \nu(M_{(\alpha)}) / \nu(G_{(\alpha)}) \to X_{(\alpha)}$, where we think of $\nu(M_{(\alpha)}) / \nu(G_{(\alpha)})$ as a stratified vector bundle in the sense of \cite{Ross}.  One may characterize the fiber of $\bar{\pi}$ as follows:
	
\begin{proposition}\label{Pro1}
The fiber of $\bar{\pi}: \nu(M_{(\alpha)}) / \nu(G_{(\alpha)}) \to X_{(\alpha)}$ at $[x]\in X_{(\alpha)}$ admits a cone identification:
$$\overline{\pi}^{-1}([x]) \simeq \frac{[0,\infty) \times (S(W_x)/G_x)}{\{0\} \times (S(W_x)/G_x)} \simeq C(S(W_x)/G_x),$$
where $S(W_x)$ is a sphere in the slice $W_x$.
\end{proposition}
\begin{proof}
	
It is well-known that $\nu(G_{(\alpha)})$ can be identified with the action groupoid $G_{(\alpha)}\ltimes \nu(M_{(\alpha)})$ (see \cite[Prop.~6]{Fernandes2014}), which guarantees that $\nu(M_{(\alpha)})/\nu(G_{(\alpha)})\simeq \nu(M_{(\alpha)})/G_{(\alpha)}$. The fiber of $\overline{\pi}$ at $[x]\in X_{(\alpha)}$ is the stratified space given by the orbit space of $\nu(G_{(\alpha)})$ restricted to $G_{\mathcal{O}_x}$. We know that $G_{\mathcal{O}_x}$ is a transitive groupoid, so it is Morita equivalent to the isotropy group $G_x$. Thus, the action groupoid $G_x\ltimes \nu_x(M,M_{(\alpha)})$ is Morita equivalent to $\nu(G_{(\alpha)})_{G_{\mathcal{O}}}$ and, consequently, we get the identification $\overline{\pi}^{-1}([x])\simeq \nu_x(M,M_{(\alpha)})/G_x$, see \cite[Cor.~3.7]{dHO}.

Let us equip $G$ with a groupoid Riemannian 2-metric (consult \cite{dHF}). This allows one to identify $\nu_x(M_{(\alpha)})$ with the vector space $W_x$, in such a way $\nu_x(\mathcal{O}_x)=\nu_x(\mathcal{O}_x)^{G_x}\oplus W_x$ and $W_x^{\perp}=\nu_x(\mathcal{O}_x)^{G_x}$. Hence, denoting $W_x^{\times} := W_x \setminus \{0\}$ and $S(W_x) := \{v \in W_x \mid \|v\| = 1\}$, we obtain the following maps:
\begin{itemize}
	\item $W_x^{\times} \to (0,\infty) \times S(W_x)$, defined by sending $v \mapsto (\|v\|, v/\|v\|)$, which is a $G_x$-equivariant diffeomorphism, and
	\item $\varphi : [0,\infty) \times (S(W_x)/G_x) \to W_x/G_x$, defined by sending $\varphi(t,[v]) = [tv]$, which induces a bijection between $(0,\infty) \times (S(W_x)/G_x)$ and $W_x^{\times}/G_x$, and collapses $\{0\} \times (S(W_x)/G_x)$ to $[0]$.
\end{itemize}

\noindent Therefore, using the above maps one obtains:
\[\overline{\pi}^{-1}([x]) \simeq W_x/G_x \simeq \frac{[0,\infty) \times (S(W_x)/G_x)}{\{0\} \times (S(W_x)/G_x)} \simeq C(S(W_x)/G_x).\]
\end{proof}

We refer to $S(W_x)/G_x$ as the \textbf{link space} at $x\in M$ and denote it by $L_x$. 

\begin{remark}
Proposition \ref{Pro1} implies that the local model around the stratum $\nu(M_{(\alpha)})/G_{(\alpha)}$ is a stratified space fibered by cones over the stratum $X_{(\alpha)}$. We know that for each $[x]\in X$ there is a neighborhood $U$ such that $U\simeq V_x/G_x$, where $V_x$ is a $G_x$-invariant neighborhood of the origin in $
\nu_x(\mathcal{O}_x)$. Decomposing $\nu_x(\mathcal{O})= 
\nu_x(\mathcal{O})^{G_x}\oplus W_x$ orthogonally as above, we get that $
\nu_x(\mathcal{O})/G_x\simeq 
\nu_x(\mathcal{O})^{G_x}\oplus W_x/G_x$. Additionally, the tangent space at the stratum $M_{(\alpha)}$ is given by $T_x(M_{(\alpha)})=T_x\mathcal{O}_x\oplus 
\nu_x(\mathcal{O})^{G_x}$, where the trivial factor of the normal representation (i.e. the subspace $
\nu_x(\mathcal{O})^{G_x}$), can be seen as the usual tangent space of $X_{(\alpha)}$ at $[x]$. This shows that the neighborhood $U$ can be identified with  $\tilde{V}	\times \tilde{C}$, where $	\tilde{V}=V_x\cap 
\nu_x(\mathcal{O})^{G_x}$ and $	\tilde{C}$ is an open neighborhood of the vertex of $C(L_x)$. See \cite[Cor.~5.8]{CM} and \cite{NV}.

We also point out that the linearization around saturated neighborhoods (compare  \cite{dHF}), implies that there exists a linearization 
$\varphi: U \subseteq \nu(G_{(\alpha)}) \to U' \subseteq G$
with $\varphi(Z) = G_{(\alpha)}$, where $Z$ is the zero section of $\nu(G_{(\alpha)})$ and $U, U'$ are some Lie groupoid neighborhoods. Hence, the induced orbit map
$\overline{\varphi} : \overline{U} \subset \nu(M_{(\alpha)})/G_{(\alpha)} \to \overline{U'} \subseteq X$ is a stratified diffeomorphism with $\overline{\varphi}(\overline{Z}) = X_{(\alpha)}$.
\end{remark}

\section{Stratified vector fields}\label{S:3}

Our aim in this section is to establish a natural relation between multiplicative vector fields on proper Lie groupoids and stratified vector fields on the corresponding orbit spaces. 

It is known that the Lie algebra of vector fields on the separated differentiable stack $[M/G]$ presented by $G\rr M$ can be modeled by means of the Lie 2-algebra\footnote{A \textbf{Lie 2-algebra} is a Lie algebra internal to the category of Lie groupoids. Equivalently, these kinds of algebraic objects can be defined by using crossed models of Lie algebras.} of multiplicative vector fields on $G$, see \cite{OW}. The key steps of this construction are as follows. We consider the Lie algebroid of $G$, which is defined to be the vector bundle $A:=\textnormal{ker}(ds)|_{M}\subset TG$ with anchor map $\rho:A\to TM$ obtained by restricting $dt:TG\to TM$ to $A$, and the Lie bracket on $\Gamma(A)$ is the one induced by the bracket of vector fields on $G$, consult \cite[Sec.~ 1.4]{crainicfernandes}. A \textbf{multiplicative vector field} on a Lie group $G\rr M$ is a pair of vector fields $(\xi,v)\in \mathfrak{X}(G)\times \mathfrak{X}(M)$ such that $\xi:G\to TG$ determines a Lie groupoid morphism covering $v:M\to TM$, see \cite{MX}. Clearly, the set of multiplicative vector field on $G$, denoted by $\mathfrak{X}_m(G)$, inhibits a Lie algebra structure in a canonical way.  A very important, and actually equivalent, property of multiplicative vector fields is that their flow determine local 1-parameter groups of Lie groupoid automorphism of $G$. For each section $a\in \Gamma(A)$ one can define right invariant $a^r$ and left invariant $a^l$ vector fields on $G$ by using the composition map of the tangent groupoid $TG\rr TM$. Indeed:
$$a^r(g)=a(t(g))\cdot 0_g\quad \textnormal{and}\quad  a^l(g)=-0_g\cdot i(a(s(g))),\qquad g\in G.$$

\noindent Following \cite{OW}, we get that there exists a crossed module of Lie algebras $(\mathfrak{X}_m(G),\Gamma(A),\delta,D)$ where the Lie algebra morphisms $\delta:\Gamma(A)\to \mathfrak{X}_m(G)$ and $D:\mathfrak{X}_m(G)\to \textnormal{Der}(\Gamma(A))$ are respectively defined by the expressions
$$\delta(a)=(a^r-a^l,\rho(a)) \quad \textnormal{and}\quad D_{(\xi,v)}(a)=[\xi,a^r]|_M,$$
for all $a\in \Gamma(A)$ and $(\xi,v)\in \mathfrak{X}_m(G)$. We refer to the corresponding Lie 2-algebra as the Lie 2-algebra of multiplicative vector fields on $G\rr M$. As a crucial feature, we have that if $G \rightrightarrows M$ and $G' \rightrightarrows M'$ are Morita equivalent Lie groupoids then the crossed modules of Lie algebras $(\mathfrak{X}_m(G),\Gamma(A),\delta,D)$ and $(\mathfrak{X}_m(G'),\Gamma(A'),\delta',D')$ are isomorphic in the derived category of crossed modules. In consequence, the following quotient spaces are isomorphic as Lie algebras:
$$\mathfrak{X}_{m}(G)/\textnormal{im}(\delta)\simeq \mathfrak{X}_{m}(G')/\textnormal{im}(\delta').$$

\noindent This enables us to define the Lie algebra of vector fields on $[M/G]$ as $\mathfrak{X}([M/G]):=\mathfrak{X}_{m}(G)/\textnormal{im}(\delta)$. Using the terminology introduced in \cite{OW}, the elements of $\mathfrak{X}([M/G])$ are called \textbf{geometric vector fields}.

Let us now define stratified vector fields on the orbit space $X$ of our proper Lie groupoid $G$. We say that a function $f:U\to \mathbb{R}$ on an open set $U \subseteq X$ is \textbf{smooth} if $\pi^\ast f \in C^{\infty}(\pi^{-1}(U))$. The algebra of smooth functions on $X$ is defined as $C^{\infty}(X)=\left\lbrace f:X \to \mathbb{R}: \pi^\ast f \in C^{\infty}(M) \right\rbrace$. Accordingly, the \textbf{sheaf of smooth functions} on $X$, denoted by $C^\infty_X$, is defined by letting $C^\infty_X(U):=C^\infty(\pi^{-1}(U)/G_{\pi^{-1}(U)})$. Observe that the pullback map $\pi^\ast: C^{\infty}(X) \to  C^{\infty}(M)$ enables us to identify the algebra of smooth functions on $X$ with the algebra of \textbf{smooth basic functions} on $M$ given by
$$C^{\infty}(M)^G=\left\lbrace f:M \to \mathbb{R}: s^\ast f=t^\ast f  \right\rbrace.$$

\noindent We denote the space of derivations of $C^{\infty}(X)$ as $\mathfrak{Der}(C^{\infty}(X))$ and refer to its elements as \textbf{smooth vector fields} on $X$. We say that a smooth vector field on $X$ is \textbf{ strata preserving} if it preserves the vanishing ideals of $C^{\infty}(X)$ induced by the strata. The set of vector fields that preserve strata will be denoted by $\mathfrak{X}^{\infty}(X)$. Geometrically, these preserving strata vector fields are tangent to each stratum. 

\begin{remark}
Since the flow of a multiplicative vector field is given by local Lie groupoid automorphisms we get that such flows preserve the connected components of the Morita types, meaning that, at the level of object, they preserve the strata of $M$. Moreover, $\mathfrak{X}_{m}(G)$ naturally acts on $C^{\infty}(M)^G\simeq C^{\infty}(X)$. 
\end{remark}

For every pair $(\xi, v) \in \mathfrak{X}_{m}(G)$ and $f \in C^{\infty}(X)$ we denote by $\pi_*(\xi, v)(f)$ the unique smooth function on $X$ such that $\pi^*(\pi_*(\xi,v)(f))=v(\pi^*(f))$. Our first main result is the following. 

\begin{theorem}\label{thm1}
	There map $\pi_*:\mathfrak{X}_{m}(G) \to \mathfrak{Der}(C^{\infty}(X))$ is a Lie algebra homomorphism  such that its image is the set of strata preserving vector fields $\mathfrak{X}^{\infty}(X)$.
\end{theorem}
\begin{proof}
	It is easy to check that $\pi_*(\xi, v)$ belongs to $\mathfrak{Der}(C^{\infty}(X))$ for all $(\xi, v) \in \mathfrak{X}_{m}(G)$ and that it preserves the bracket. Let us show that it preserves the strata of $X$. From \cite[Prop.~23]{CM} we know that for each orbit $\mathcal{O}_x$ (i.e. point in $X$) there is a neighborhood $U$ in $X$ such that $U\simeq W/G_x$ and $C^{\infty}(U)\simeq C^{\infty}(W)^{G_x}$, where $W$ is an open neighborhood in $\nu_x(\mathcal{O}_x)$, centered at the origin, which is invariant by the normal representation. Furthermore, $X_{(\mathcal{O}_x)}\cap U\simeq W^{G_x}$ and $M_{(x)}\cap \pi^{-1}(U)\simeq \mathcal{O}_x\times W^{G_x}$. It follows that if $f \in C^{\infty}(U)$ vanishes along $X_{(\mathcal{O})}\cap U$ then $\pi^*(f)$ vanishes along $M_{(x)}\cap \pi^{-1}(U)$. Hence, since $v$ preserves the strata we get that $v(\pi^*(f))$ vanishes along $M_{(x)}\cap \pi^{-1}(U)$, and, consequently, $\pi_*(\xi,v)(f)$ vanishes along $X_{(\mathcal{O}_x)}\cap U$.

	To prove that $\pi_*$ is surjective onto the set of stratified vector fields, let us pick $Z\in\mathfrak{X}^{\infty}(X)$. Since $G$ is proper, we can consider an open cover
	$\mathcal{U}=\left\lbrace U_{\lambda}\right\rbrace_{\lambda\in \Lambda}$
	of $X$ such that, for every $\lambda\in \Lambda$, the groupoid $G_{\pi^{-1}(U_{\lambda})}$ is Morita equivalent to the action groupoid $G_x\ltimes \nu_x(\mathcal{O}_x)$ associated with the normal representation at a point $x$ in $\pi^{-1}(U_\lambda)$, see \cite[Prop.~4.7]{CS}. Let us now consider a partition of unity
    $\left\lbrace \phi_{\lambda}\right\rbrace_{\lambda \in \Lambda}$
    		subordinated to $\mathcal{U}$, consult \cite[Prop.~8]{CM}. If we verify that for each $\lambda \in \Lambda$ there exists $(\xi^{\lambda}, v^{\lambda}) \in \mathfrak{X}_m\bigl(G_{\pi^{-1}(U_{\lambda})}\bigr)$ such that $\pi_*(\xi^{\lambda}, v^{\lambda})=Z_{U_{\lambda}}$, then we can build the multiplicative vector field $Z^*=\sum_{\lambda}(\pi \circ t)^*(\phi_{\lambda})(\xi^{\lambda},v^{\lambda})$
    		on $G$, so that $\pi_*(Z^*)=Z$.
	Choose $U_\lambda \in \mathcal{U}$ and let $P_x$ be a principal bi-bundle realizing the Morita equivalence between $G_{\pi^{-1}(U_{\lambda})}$ and $G_x\ltimes \nu_x(\mathcal{O}_x)$. We denote by $\epsilon:P_x\to \pi^{-1}(U_{\lambda})$ and $\eta:P_x\to \nu_x(\mathcal{O}_x)$ the submersions associated with $P_x$. Thus, it follows that $G_{\pi^{-1}(U_{\lambda})}$ is isomorphic to $\mathrm{Gauge}(P_x)$, where the gauge groupoid is considered as the quotient of the submersion groupoid $P_x\times_{\eta}P_x \rightrightarrows P_x$ under the free and proper action of $G_x\ltimes \nu_x(\mathcal{O}_x)$ along the submersion $\eta$. Let us fix a neighborhood $W \subseteq \nu_x(\mathcal{O}_x)$ as above. Using the equivariant smooth lifting theorem (see \cite[Sec.~3]{GSchwarz}), we obtain a vector field $Z_{\lambda}^* \in \mathfrak{X}(W)^{G_x}$ that projects onto $Z_{U_{\lambda}}$. Since $G_x$ is compact, it is possible to choose a $G_x$-invariant Ehresmann connection for the submersion $\eta$. Therefore, we can lift the invariant vector field $Z_{\lambda}^*$ to a $G_x$-invariant, $\eta$-projectable vector field $\tilde{Z}_{\lambda}$ on $P_x$. This construction yields a $G_x$-invariant multiplicative vector field on the submersion groupoid $P_x\times_{\eta}P_x \rightrightarrows P_x$. Hence, after projecting it onto $\mathrm{Gauge}(P_x)$, we get the desired multiplicative vector field $(\xi^{\lambda},v^{\lambda})$ on $G_{\pi^{-1}(U_{\lambda})}$.
\end{proof}

An important consequence of Theorem \ref{thm1} is that geometric vector fields on the separated differentiable stack $[M/G]$ are the same as stratified vector field on the orbits space $X$. More precisely, we can state and prove the following result.

\begin{proposition}\label{prop:exact-sequence}
	Let $G\rr M$ be a proper Lie groupoid and let $(\mathfrak{X}_m(G),\Gamma(A),\delta,D)$ denote its crossed module of multiplicative vector fields. Then, the sequence 
	\begin{center}
		\begin{tikzcd}
			0\arrow[r]&\mathrm{im}(\delta)\arrow[r, hook]&\mathfrak{X}_{m}(G)\arrow[r,"\pi_*"]&\mathfrak{X}^{\infty}(X)\arrow[r]&0,
		\end{tikzcd}
		
	\end{center}
	is exact. Consequently, geometric vector fields on the separated differentiable stack $[G/M]$ presented by $G\rr M$ give rise to stratified vector fields on the orbit space $X$ and there is a Lie algebra isomorphism $\mathfrak{X}([M/G])\simeq \mathfrak{X}^{\infty}(X)$.  
\end{proposition}
\begin{proof}
	First of all, it is clear that $(\xi,v) \in \ker(\pi_*)$ if and only if for any $G$-invariant (i.e. basic) function $f:M\to \mathbb{R}$ it holds $\mathcal{L}_{v}(f)=0$, so that $\mathrm{im}(\delta)\subseteq \ker(\pi_*)$. Let us check that for any $(\xi,v) \in \ker(\pi_*)$ there exists $a\in \Gamma(A)$ such that $\xi=a^r-a^l$ and $v=\rho(a)$. One can consider the exact sequence of vector bundles
	\[ \begin{tikzcd}
		0\arrow[r]&t^*A\arrow[r,"r"]&TG \arrow[r,"ds"]& s^*TM\arrow[r]&0,
	\end{tikzcd}\]
	where the map $r$ relates to any section of $A$ its associated right invariant vector field on $G$. Let $\bar{v}\in \Gamma(s^*TM)$ be such that $\bar{v}=s\circ v$ and let $\bar{v}^h \in \mathfrak{X}(G)$ denote the lifting of $\bar{v}$ by a fixed splitting of the above sequence. For $\xi-\bar{v}^h \in \mathfrak{X}(G)$ one has that $ds(\xi-\bar{v}^h)=v\circ s-\bar{v}=0$. Therefore, the exactness of the sequence guarantees the existence of $a \in \Gamma(A)$ such that $a^r=\xi-\bar{v}^h$. By our initial assumption, it follows that $v$ is tangent to the groupoid orbits and, consequently, $\bar{v}^h \in \Gamma(\ker(dt))$. Then: \[0=dt(\bar{v}^h)=dt(\xi-a^r)=v\circ t-\rho(a)\circ t,\]
	thus obtaining that $\rho(a)=v$. Recall that $t=s\circ i$, so that $ds(di(\bar{v}^h))=0$. This implies that there is $b\in \Gamma(A)$ such that $b^r=di(\bar{v}^h)$ and hence $\xi=a^r+di(b^r)$. Since $\xi$ is a multiplicative vector field and both $a^r$, $b^r$ are right invariant vector fields one deduces that the pair of vector fields $(\xi, \xi)$, $(a^r,0)$, and $(b^r,0)$ on $G^{(2)}$ are $m$-related to $\xi$,  $a^r$, and $b^r$ on $G$, respectively. If we rewrite $$(\xi,\xi)=(a^r,0)+(di(b^r),a^r)+(0,di(b^r)),$$ 
	then we easily see that $(di(b^r),a^r)$ is $m$-related to $0$. But 
	$a^r=dm(b^r, dm(di(b^r),a^r))=dm(b^r,0)=b^r$ and therefore $\xi=a^r+di(b^r)=a^r-a^l$. 
\end{proof}


We now turn to some applications of Theorem \ref{thm1}. In this direction, we first establish a stacky version of the generalized Gauss lemma, a result known to have several applications in geometric analysis \cite[Chap.~2]{Gray}. To this end, we begin by introducing some terminology. A Riemannian metric on $X$ can be thought of as an equivalence class of a groupoid Riemannian metric on $G\rr M$ in the sense of del Hoyo and Fernandes \cite{dHF,dHF2}. The manifolds of arrows $G$ and objects $M$ respectively inherit Riemannian metrics $\eta^{(1)}$ and $\eta^{(0)}$ such that $s$ and $t$ become Riemannian submersions and the inversion $i$ gives rise to an isometry. Additionally, the induced bundle metrics along the normal directions to the orbits are such that the normal isotropy representations are by linear isometries. This provides us with a way of inducing ``inner products'' over the ``coarse'' tangent spaces $T_{[x]} X \approx \nu_x(\mathcal{O}_x)/G_x$ for all $x\in M$. In particular, this allows us to measure the length of stacky curves and distances on $X$.

Following \cite[Thm.~6.1]{Pflaum}, we have that the geodesic distance $d_N$ on $X$ can be obtained as
$$d_N([x],[y])=\inf\lbrace d(x_1,\mathcal{O}_x)+\cdots+d(x_k,\mathcal{O}_{x_{k-1}})\rbrace,$$
where $d$ stands for the geodesic distance associated with $\eta^{(0)}$ on $M$, and the infimum is taken over all choices of pairs $(x_j,\mathcal{O}_{x_j})$ for $j=1,\dots,k-1$ and $x_k\in \mathcal{O}_y$, over all $k\in \mathbb{N}$.
Additionally, the distance $d_N$ has the following properties:
\begin{itemize}
	\item it is uniquely determined by the property that for each orbit $\mathcal{O}$ in $M$ and every point $z\in T_{\mathcal{O}}$ of an appropriate metric tubular neighborhood $T_{\mathcal{O}}$ of $\mathcal{O}$, the relation $d_N(\mathcal{O},\mathcal{O}_z)=d(z,\mathcal{O})$ holds true, where $\mathcal{O}_z$ is the orbit through $z$,
	\item the canonical projection $\pi:M\to X$ is a submetry, and
	\item the topology induced on $X$ by $d_N$ coincides with the quotient topology with respect to $\pi$.
\end{itemize}

\noindent It is worth mentioning that $d_N$ admits a more geometric interpretation, which is given by expression:
\[
d_N([x],[y]) = \inf \left\{ \ell(\alpha) : \alpha \colon [0,1] \to [M/G], \ \alpha(0) = [x], \ \alpha(1) = [y] \right\},
\]
where $\ell(\alpha)$ denotes the length of a stacky curve $\alpha$ in $[M/G]$ with respect to the stacky metric $[\eta]$ and the infimum is taken over all such curves joining $[x]$ to $[y]$, see \cite[Thm.~3]{dHdM1}. In particular, the latter characterization of $d_N$ implies that it is invariant under Morita Riemannian equivalences.

In these terms, we have the following stacky version of the Gauss lemma.

 \begin{proposition}\label{Pro2}

Let $G \rightrightarrows M$ be a Riemannian groupoid. Then, for any $x \in M$ such that the orbit $\mathcal{O}_x$ is an embedded submanifold, there exists $\epsilon > 0$ verifying that:
\begin{enumerate}
	\item the distance function $r_x \colon B^{\circ}_{d_N}([x],\epsilon) \to \mathbb{R}$, defined by $r_x([y]) = d_N([x],[y])$, is smooth on the punctured neighborhood $B^{\circ}_{d_N}([x],\epsilon)$, and
	\item its gradient $\nabla r_x:=\pi_*(\nabla r_{\mathcal{O}_x})$ is a stratified vector field on $B^{\circ}_{d_N}([x],\epsilon)$.
\end{enumerate}

\end{proposition}

\begin{proof}

Let $B(0,\epsilon)\subset \nu_x(\mathcal{O}_x)$ be the ball of normal vectors at $x$ of radius $\epsilon$. If $\epsilon$ is sufficiently small, then $S=\exp\big(B(0,\epsilon)\big)$ is a slice at $x$. Let $U$ be the open set obtained as the saturation of $S$. It follows that the Lie groupoids $G_S \toto S$ and $G_U \toto U$ inherit Riemannian structures induced by the ambient Lie groupoid $G$, and hence their own normal distances. Furthermore, with respect to these induced structures, the natural inclusion $G_S \hookrightarrow  G_U$ becomes a Riemannian Morita map, so the induced map $(S/G_S,d_N)\to(U/G_U,d_N)$ preserves the geodesic distances (consult \cite[Cor.~3]{dHdM1}). In particular,
$d_N([x]_S,[y]_S)=d_N([x],[y])$, where $[x]_S$ corresponds to the element $[x]$ seeing inside $S/G_S$.

We have thus reduced the problem to the fixed-point case.
After possibly shrinking $\epsilon$, we can assume that the exponential map yields a groupoid linearization as in \cite[Thm.~5.11]{dHF}, namely,
\[
\exp : (G_x \ltimes B(0,\epsilon) \toto B(0,\epsilon)) \to (G_S \toto S).
\]
Since the action $G_x \act \nu_x(\mathcal{O}_x)$ preserves the norm, the orbits of $G_S$ are contained in the geodesic spheres around $x$. Moreover, by the classical Gauss lemma (see \cite[Lem.~2.11]{Gray}), we have that $d(x,y)=\|v\|$, where $y=\exp_x(v)$. Therefore:
\[
r_{x}([y]) = d_N([x],[y]) = d_N([x]_S,[y]_S) = \|v\|,
\]
which is a smooth function on $U/G_U - \{[x]\}$. Using the extended model, we can realize $r_{[x]}$ as a basic function $r_{\mathcal{O}_x}$ on $U - \mathcal{O}_x$ with corresponding function $R_{\mathcal{O}_x}=s^\ast (r_{\mathcal{O}_x})$ on $G_U - G_{\mathcal{O}_x}$. Therefore, as the gradient vector field $\nabla R_{\mathcal{O}}$ is multiplicative (compare \cite[Prop.~ 6.3]{Valencia}), we get that $\nabla r_x:=\pi_*(\nabla r_{\mathcal{O}_x})$ is a stratified vector field on $B^{\circ}_{d_N}([x],\epsilon)$, as desired.

\end{proof}

As every proper Lie groupoid admits a Riemannian groupoid metric \cite[Thm.~4.13]{dHF}, the previous result applies to any Lie groupoid of this type.

\section{Lifting smooth isotopies}\label{S:4}

A celebrated result due to Schwarz asserts that a smooth version of Palais’ covering homotopy theorem holds for actions of compact Lie groups \cite{GSchwarz} (see also \cite{Davis}). Kankaanrinta later extended this result to the case of proper actions of non-compact Lie groups \cite{Kankaanrinta}. Building on \cite{Davis,Kankaanrinta,GSchwarz}, in this section we partially generalize both Schwarz’s and Kankaanrinta’s results by establishing an isotopy lifting-type result for orbit spaces of certain proper Lie groupoids. This may be regarded as another application of Theorem \ref{thm1}, one that warrants a separate section.

Let $G \rightrightarrows M$ and $H \rightrightarrows N$ be proper Lie groupoids with orbit spaces $X$ and $Y$, respectively. A map $\psi:X \to Y$ is said to be \textbf{smooth} if $\psi^\ast C^\infty(Y) \subseteq C^\infty(X)$, meaning that the pullback of any smooth function on $Y$ along $\psi$ is a smooth function on $X$. For each $[x]\in X$, we denote by $\mathcal{R}_{[x]}$ the ring of germs of smooth functions at $[x]$ vanishing at $[x]$. Accordingly, $\mathcal{R}_{[x]}^2$ stands for the ideal of $\mathcal{R}_{[x]}$ which is spanned by the products of elements of $\mathcal{R}_{[x]}$. In these terms, the \textbf{Zariski tangent space} $T_{[x]}^ZX$ of $X$ at $[x]$ is defined as the dual vector space of $\mathcal{R}_{[x]}/\mathcal{R}_{[x]}^2$. It follows from the linearization theorem that $ X $ locally looks like the orbit space of a compact Lie group representation. Therefore, the vector space $T_{[x]}^ZX$ is finite-dimensional, as the Zariski tangent space associated with the orbit space of the normal isotropy representation are always finite-dimensional, consult \cite[p.~44]{GSchwarz}. Besides, it is simple to see that any smooth map $ \psi : X \to Y $ induces a linear map on tangent spaces $ (d\psi)_{[x]} : T_{[x]}^Z(X) \to T_{\psi([x])}^Z(Y) $.

Two notions of normal transversality can be defined as follows. On the one hand, the \textbf{Zariski normal space} at $[x]\in X$ is defined to be $ N_{[x]}^Z(X) := T_{[x]}^ZX / T_{[x]}X_{[x]}$, where $T_{[x]}X_{[x]}$ is the usual tangent space of $X_{[x]}$ at $[x]$. Here $ X_{[x]} $ stands for the Morita type stratum of $ X $ containing $[x]\in X$. If $\psi : X \to Y$ is a strata-preserving smooth map, i.e. $\psi$ yields isomorphisms when restricted to isotropies and sends strata of $ X $ to strata of $ Y $, then $ (d\psi)_{[x]} $ induces a linear map on normal spaces $ (d \psi)_{[x]} : N_{[x]}^Z(X) \to N_{\psi([x])}^Z(Y) $, which we keep denoting by $ (d\psi)_{[x]}$. We say that $\psi$ is \textbf{Zariski normally transverse} if $ (d \psi)_{[x]} $ is an isomorphism between the Zariski normal spaces for all $ [x] \in X $. On the other hand, we define the \textbf{stacky normal space} at each $ x \in M $ as
$N_{[x]}(X):= \nu_x(\mathcal{O}_x) / \nu_x(\mathcal{O}_x)^{G_x},$ where $\nu_x(\mathcal{O}_x)^{G_x}$ stands for the space of fixed normal vectors with respect to the isotropy representation of $G_x$. It is clear that the base map of any Lie groupoid morphism $ \psi : (G\rightrightarrows M) \to (H\rightrightarrows N)$ preserves the Morita type stratifications for $X$ and $Y$, thus inducing a smooth strata-preserving map $\overline{\psi}:X\to Y$. We call $\psi$ \textbf{stacky normally transverse} if the induced linear map $[d\overline{\psi}]_{[x]} : N_{[x]}(X) \to N_{[\psi(x)]}(Y)$ is an isomorphism for every $[x]\in X$. 

\begin{lemma}\label{LemmaLif1}
	Let $\psi : (G\rightrightarrows M) \to (H\rightrightarrows N)$ be a Lie groupoid morphism. Then, the following assertions hold true:
	\begin{enumerate}
		\item Pick $x \in M$ such that the induced linear map $d\psi_{x}: \nu_x(\mathcal{O}_x) \to \nu_{\psi(x)}(\mathcal{O}_{\psi(x)})$ is surjective. Then, one can choose a slice $S$ at $x$ so that $\tilde{S} := \psi(S)$ is a slice at $\psi(x)$.
		\item If the induced linear map $d\psi_x:\nu_x(\mathcal{O}_x) \to \nu_{\psi(x)}(\mathcal{O}_{\psi(x)})$ is surjective and the induced Lie group morphism $\psi_x: G_x \to H_{\psi(x)}$ is a submersion, then  $\psi$ can be linearized in some neighborhood of $x$.
		\item Suppose that both $d\psi_x:\nu_x(\mathcal{O}_x) \to \nu_{\psi(x)}(\mathcal{O}_{\psi(x)})$ and $\psi_x: G_x \to H_{\psi(x)}$ are isomorphisms. Then, there exist neighborhoods $U$ of $\mathcal{O}_x$ in $X$ and $V$ of $\mathcal{O}_{\psi(x)}$ in $Y$ such that $\overline{\psi}:U \subseteq X \to V \subseteq Y$ is a diffeomorphism.
	\end{enumerate}
\end{lemma}

\begin{proof}
	Since $G$ is proper, there exists a relatively compact slice $S^0$ at $x\in M$. Shrinking $S^0$ if necessary, we can assume that $d\psi|_{S^0}$ has the same rank as $d\psi_x$. By the standard Constant Rank Theorem, there is a neighborhood $S^1 \subseteq S^0$ of $x$ such that $\tilde{S}^0 := \psi(S^1)$ is an embedded submanifold of $N$. Note that the surjectivity of $d\psi_x$ implies that  
	$T_{\psi(x)} N = T_{\psi(x)} \tilde{S}^0\oplus T_{\psi(x)} \mathcal{O}_{\psi(x)}$. Consequently, there exists a neighborhood $\tilde{S}^1 \subseteq \tilde{S}^0$ such that $T_y N = T_y \tilde{S}^1\oplus T_y \mathcal{O}_y$ for all $y \in \tilde{S}^1$. As $\tilde{S}^1$ is relatively compact and $H$ is proper, the orbit $\mathcal{O}_y$ intersects $\tilde{S}^1$ in a compact subspace. Therefore, there exists a  smaller neighborhood $\tilde{S}^2 \subseteq \tilde{S}^1$ such that $\tilde{S}^2 \cap \mathcal{O}_y = \{y\}$, meaning that $\tilde{S}^2$ is a slice at $y$. We thus obtain the desired slice by setting $S = S^0 \cap \psi^{-1}(\tilde{S}^2)$.
	
	We now turn to the proof of item (2). Consider a slice $S$ at $x\in M$ as in item (1). Pick $g \in G_x$, so that we get a decomposition 
	$T_g (G_{S}) \simeq T_g G_x \oplus T_x S_x$. Hence, the differential of $\psi$ induces a map $d\psi_g : T_g G_x \oplus T_x S \;\longrightarrow\; T_{\psi(g)} H_{\psi(x)} \oplus T_{\psi(x)} \tilde{S}$ which preserves the splitting above. Since $G_x$ is compact, we may shrink $S$ if necessary, in such a way that $\psi|_{S} : G_{S} \longrightarrow H_{\tilde{S}}$ is a fibration. Thus, from \cite[Thm.~4.2.3]{dHF2}, it follows that there exist invariant neighborhoods of $x$ in $S$ and of $\psi(x)$ in $\tilde{S}$ on which the restriction $\psi|_S$ is linearizable. The required linearization of $\psi$ can be obtained by considering local linearizations of $G$ at $x$ and of $H$ at $\psi(x)$ as in \cite[Cor.~3.11]{Pflaum}, which in turn may be chosen to be compatible with the slices $S$ and $\tilde{S}$. 
	
	 Finally, item (3) directly follows by combining the standard Inverse Function Theorem and item (2). 
\end{proof}

We say that a Morita map $G \to H$ is a \textbf{Morita fibration} if the corresponding map on objects $M\to N$ is a surjective submersion.

\begin{proposition}\label{morita-->stackynormal}
	Let $\psi : (G\rightrightarrows M) \to (H\rightrightarrows N)$ be a Morita fibration. Then:
	\begin{enumerate}
		\item $\psi^0:M\to N$ is Zariski normally transverse.
		\item The restriction $\overline{\psi}|:X_{[x]}\to Y_{[\psi(x)]}$ is a diffeomorphism for all $[x]\in X$.
		\item For every $[x]\in X$ there are open neighborhoods $U$ of $X_{[x]}$ in $X$ and $U'$ of $Y_{[\psi(x)]}$ in $Y$ such that $\overline{\psi}:U\subseteq X\to U'\subseteq Y$ is a diffeomorphism.
		\item $\psi$ is stack normally transverse.
	\end{enumerate}
\end{proposition}
\begin{proof}
	From \cite[Thm.~4.3.1]{dH} it follows that $\psi$ preserves the transversal information of $G$ and $H$, so that the induced maps $\overline{d\psi}_x:\nu_x(\mathcal{O}_x)^{G_x}\to \nu_{\psi(x)} (\mathcal{O}_{\psi(x)})^{H_{\psi(x)}}$ and  $\overline{d\psi}_x:\nu_x(\mathcal{O}_x)/\nu_x(\mathcal{O}_x)^{G_x}\to \nu_{\psi(x)}(\mathcal{O}_{\psi(x)})/\nu_{\psi(x)}(\mathcal{O}_{\psi(x)})^{H_{\psi(x)}}$ are linear isomorphisms. Since, for manifolds, the usual tangent space and the Zariski tangent space are isomorphic, it follows that $\psi^0:M\to N$ is Zariski normally transverse. Furthermore, from \cite[Lem.~4.27]{CM} we get that $T_x M_{(x)}=T_x\mathcal{O}_x\oplus \nu_x(\mathcal{O}_x)^{G_x}$, so $d\pi_x:\nu_x(\mathcal{O}_x)^{G_x}\to T_{\pi(x)}X_{(\pi(x))}$ is a linear isomorphism. Therefore, the linear map $d\overline{\psi}_{[x]}:T_{[x]}X_{[x]}\to T_{[\psi(x)]}Y_{[\psi(x)]}$ fits into a commutative diagram in which all the other arrows are linear isomorphisms, thus obtaining that it is also an isomorphism. This proves (1).
	
We know that $\overline{\psi}:X\to Y$ is a homeomorphism. In consequence, the standard Inverse Function Theorem implies that its restriction to every stratum is a diffeomorphism, and (2) follows.

Let us now show (3). Following \cite[Thm.~4.2.3]{dHF2}, one can consider linearizations $\varphi:W\subseteq \nu(G_{(x)})\to V\subseteq G$ and $\varphi':W'\subseteq \nu(H_{(\psi(x))})\to V'\subseteq H$ such that $\psi \circ \varphi=\varphi' \circ \nu(\psi)$, where $V$ and $V'$ are saturated neighborhoods of $G_{(x)}$ and $H_{(\psi(x))}$, respectively. Observe that $\nu(\psi):\nu(G_{(x)})\to \nu(H_{(\psi(x))})$ is a $VB$-groupoid morphism covering $\psi|_{M_{(x)}}:G_{(x)}\to H_{(\psi(x))}$, the normal bundle $\nu(M_{(x)}) = TM/TM_{(x)} \simeq \nu(\mathcal{O}_x) / \nu(\mathcal{O}_x)^{G_{\mathcal{O}_x}}$, and $\nu(M,M_{(x)})/ \nu(G,G_{(x)})$ gives rise to a stratified vector bundle over $X_{[x]}$, compare Proposition \ref{Pro1} and \cite{Ross}. Therefore, the linearization of a Morita fibration around a stratum $\nu(\psi)$ turns out to be a fiberwise isomorphism, contrast \cite[Prop.~6.4.1]{dHF2}. From (2), it follows that the induced map $\overline{\nu(\psi)}: \nu(M_{(x)})/ \nu(G_{(x)})\to \nu(N_{(\psi(x))})/ \nu(H_{(\psi(x))})$ covers the diffeomorphism $\overline{\psi}|:X_{[x]}\to Y_{[\psi(x)]}$. Hence, the map $\overline{\psi}|:U\subseteq X \to U'\subseteq Y$ is a diffeomorphism, where $U:=\pi(V)$ and $U'=\pi(V')$.

Finally, (4) is obtained as a consequence of (3).
\end{proof}

While we do not have a proof, the previous results suggest that smooth diffeomorphisms between orbit spaces arise from stacky normally transverse morphisms between the corresponding proper Lie groupoids.

We are now in conditions to state our second main result.

\begin{theorem}\label{thm2}
Let $\psi: (G \rightrightarrows M) \to (H \rightrightarrows N)$ be a Lie groupoid morphism and let $H$ be a Lie groupoid such that the canonical projection $\pi : N \to Y$ is proper. Let $F : \mathbb{R} \times X \to Y$ be a stratified map such that the map $F_\tau:=F|_{\{\tau\}\times X}$ is a normally transverse embedding for every $\tau\in \mathbb{R}$ and $F_0 = \overline{\psi}:X\to Y$. Then, the following assertions hold true:
\begin{enumerate}
    \item There exists a Lie groupoid morphism $\Psi : \mathbb{R} \times G \to H$ such that $\Psi_0 = \psi$ and the induced map between orbit spaces $\overline{\Psi}$ agrees with $F$.
    \item If $\psi$ is a Morita fibration then $\Psi_\tau$ is also a Morita fibration for all $\tau\in \mathbb{R}$. 
\end{enumerate}
\end{theorem}
\begin{proof}
The strategy of the proof goes as follows. First, we describe the infinitesimal counterpart of $F$, which is determined by $F_0 : X \to Y$ together with a one-parameter family of stratified vector fields $\lbrace V_\tau^F\rbrace_{\tau\in \mathbb{R}} \subset \mathfrak{X}^{\infty}(Y)$. Next, we consider a lifting of each $V_\tau^F$ to a multiplicative vector field $Z_\tau^F$ on $H$. Finally, we integrate both $\psi$ and $Z_\tau^F$ to obtain $\Psi$.

The one-parameter family of stratified vector fields $\lbrace V_\tau^F\rbrace_{\tau\in \mathbb{R}}$ is determined by considering the suspension map of $F$, which is denoted by $\hat{F}:\mathbb{R}\times X\to \mathbb{R}\times Y$ and defined by  $\hat{F}(\tau,x)=(\tau,F_{\tau}(x))$. Observe that the hypothesis that $F_\tau$ is normally transverse implies that $\hat{F}$ is also normally transverse. It then follows from the Inverse Function Theorem \cite[Thm.~1.11]{GSchwarz} that $\hat{F}$ is a diffeomorphism onto its image $\hat{F}(\mathbb{R} \times X)$. Since $\frac{\partial}{\partial \tau}\oplus 0$ is clearly a stratified vector field on $\mathbb{R}\times X$, the vector field $\hat{V}:=\hat{F}_*(\frac{\partial}{\partial \tau}\oplus 0)$ defined as

$$\hat{V}(f)(\tau,F_{\tau}(x))=\frac{\partial f}{\partial \tau}(\tau,F_{\tau}(x)),\qquad f\in C^{\infty}(\hat{F}(\mathbb{R}\times X)),$$ 
is also a stratified vector field on $\hat{F}(\mathbb{R}\times X)$. By Theorem \ref{thm1}, there is a multiplicative vector field $Z$ on $(\mathbb{R}\times H)_{(\textnormal{id}\times \pi)^{-1}(\hat{F}(\mathbb{R}\times X))}$ such that $(\textnormal{id}\times \pi)_*(Z)=\hat{V}$. Observe that both $\hat{V}$ and $Z$ are vector fields related to $\frac{\partial}{\partial \tau}$ via the first component projection. Hence, it follows from \cite[Lem.~6.6]{CMS} that one can assume $Z$ to be globally defined over $\mathbb{R}\times H$, and consequently $\hat{V}$ is also globally defined on $\mathbb{R}\times Y$. In particular, we can rewrite $\hat{V}=\frac{\partial}{\partial \tau}\oplus V_{\tau}^F$ and $Z=\frac{\partial}{\partial \tau}\oplus Z_{\tau}$, where $V_\tau^F\in \mathfrak{X}^{\infty}(Y)$ and $Z_\tau\in \mathfrak{X}_m(H)$ are time-dependent vector fields such that $\pi_*(Z_\tau)=V_{\tau}^F$.

 From \cite[Prop.~4.2.33]{AMR}, we know that the local flow of $Z$ is given by $\varphi_\tau^Z(\varsigma ,p)=(\tau+\varsigma ,\Phi_{\tau+\varsigma ,\varsigma }(p))$, where $\Phi_{\tau,\varsigma }$ is the evolution operator of $Z_\tau$. Since $\hat{V}$ is $\hat{F}$-related to the complete vector field $\frac{\partial}{\partial \tau} \oplus 0$, with flow given by $\hat{F} \circ \varphi_\tau^{\frac{\partial}{\partial \tau}\oplus 0} \circ \hat{F}^{-1}$, it is itself complete. From the completeness of $\hat{V}$ and the fact that $\textnormal{id} \times \pi \colon \mathbb{R} \times N \to \mathbb{R} \times Y$ is a proper map, we conclude that $Z$ is also complete over  $(\mathbb{R} \times H)_{(\textnormal{id} \times \pi)^{-1}(\hat{F}(\mathbb{R} \times X))}$.

 We finish the proof by noting that $\Phi_{\tau,0}$ is groupoid automorphism, since $\Phi_{\tau,0}=\pi_2 \circ \varphi_{\tau}^Z \circ \iota_0$, where $\iota_0(g)=(0,g)$ and $\pi_2(t,g)=g$ for all $g\in G$ and $\tau\in \mathbb{R}$. Therefore, we define the desired groupoid morphism by setting $\Psi_{\tau}:=\Phi_{\tau,0}\circ \psi$, which proves assertion (1). In the case where $\psi$ is a Morita fibration we obtain that $\Psi_{\tau}$ is the composition of a Morita fibration with a groupoid automorphism. 
\end{proof}

\begin{remark}
A couple of observations are in order. On the one hand, it is worth noting that Theorem \ref{thm2} partially extends the homotopy lifting theorem due to Schwarz in \cite{GSchwarz} (see also \cite{Davis,Kankaanrinta}), since we only lift maps on orbit spaces that are homotopic through isotopies. On the other hand, the homotopy lifting property is a common feature of locally trivial fibrations. In fact, Ehresmann’s theorem states that proper submersions between manifolds are locally trivial fibrations, and Thom’s first isotopy lemma \cite[Prop.~11.1]{Mather} ensures that stratified proper submersions are locally trivial. This allows us to conclude that assuming our proper groupoid has a proper orbit projection is not very restrictive. A stacky version of Ehresmann’s theorem can be found, for instance, in \cite{dHF2}.
\end{remark}

Let us now exhibit some consequences of our lifting result. We say that two stratified automorphisms $\phi_0,\phi_1:X\to X$ are \textbf{related}, and we write $\phi_0\simeq \phi_1$, if there is a stratified map $F:\mathbb{R}\times X\to X$ such that the map $F_\tau:=F|_{\{\tau \}\times X}$ is a stratified diffeomorphism for every $\tau\in\mathbb{R}$ and it satisfies $F_0=\phi_0$ and $F_1=\phi_1$. We denote by $\mathrm{Diff}_{0}^{\infty}(X)$ the set of stratified automorphisms of $X$ that are related to the identity and $\mathrm{Diff}_{\mathrm{lift}}^{\infty}(X)$ the group of stratified automorphisms of $X$ that lift to groupoid automorphisms of $G$. Note that Theorem \ref{thm2} implies that $\mathrm{Diff}_{0}^{\infty}(X)$ is a subgroup of $\mathrm{Diff}_{\mathrm{lift}}^{\infty}(X)$.

A \textbf{bisection} of $G\rightrightarrows M$ is a smooth map $\sigma: M \to G$ such that $s\circ\sigma = \mathrm{id}_M$ and $t\circ\sigma: M\to M$ is a diffeomorphism. The set of all bisections of $G$ is denoted by $\mathrm{Bis}(G)$. Such a set turns out to be an infinite-dimensional Lie group (see \cite{SchmedingWockel}), where the product operation given by

$$(\sigma\bullet \sigma')(x) := \sigma((t\circ \sigma')(x))*\sigma'(x), \quad \sigma,\sigma'\in \mathrm{Bis}(G), \quad x\in M.$$

\noindent Each $\sigma \in \mathrm{Bis}(G)$ canonically determines an inner automorphism of $G$, covering $t\circ \sigma$, which is defined by $I_\sigma(g):=\sigma(t(g))*g*i(\sigma(s(g)))$ for all  $g\in G$. In particular, if $\mathrm{Aut}(G)$ stands for the group of automorphism of $G$, then $I:\mathrm{Bis}(G)\to \mathrm{Aut}(G)$ is a homomorphism of groups. Additionally, every $\Phi\in \mathrm{Aut}(G)$ defines a group automorphism of $\mathrm{Bis}(G)$, which is given by $\alpha_{\Phi}(\sigma)=\Phi_1 \circ \sigma \circ \Phi_0^{-1}$ for all $\sigma \in \mathrm{Bis}(G)$. It follows that the quadruple $(\mathrm{Aut}(G),\mathrm{Bis}(G),I,\alpha)$ yields a crossed module of infinite dimensional Lie groups\footnote{The reader is referred to \cite[Sec.~2.2]{HCV23} for a brief review of crossed modules of Lie groups and their relationship with Lie 2-groups.} which completely determines the Lie 2-group of automorphisms of $G$. It is also important to mention that the infinitesimal counterpart of $(\mathrm{Aut}(G),\mathrm{Bis}(G),I,\alpha)$ is the crossed module of Lie algebras $(\mathfrak{X}_m(G),\Gamma(A),\delta,D)$. Especially, if we consider the Lie group exponential $\exp:\Gamma(A)_c\to \mathrm{Bis}(G)$, that is defined by $\exp(a)(x)=\varphi_{1}^{a^r}(1_x)$ for all $x\in M$, and denote the flow of a vector field $X$ by $\varphi_{\tau}^X:=\mathrm{Exp}(\tau X)$, then we get the formula  $\mathrm{Exp}(\tau\delta(a))=I_{\exp(\tau   a)}$ for all $a \in \Gamma(A)$.

The following auxiliary formulas will be used below.

\begin{lemma}\label{product-identities}
Let $(K,R,\rho,\alpha)$ be a crossed module of Lie groups. Then, for each $n\in \mathbb{N}$ the following formulas are satisfied:
	\begin{enumerate}
		\item $k^n\rho(r)=\rho(\alpha_{k^n}(r))k^n$, and
		\item $(k\rho(r))^n=\rho(\prod_{l=1}^n\alpha_{k^l}(r))k^n$
	\end{enumerate}
for all $k\in K$ and $r\in R$.
\end{lemma}

In these terms, we have the following consequence of Theorem \ref{thm2}.

\begin{proposition}\label{cor1}
	Let $G\rightrightarrows M$ be a proper Lie groupoid such that its canonical projection $\pi:M\to X$ is a proper map. Then:
	\begin{enumerate}
	\item If $\phi_0 \in \mathrm{Diff}_{\mathrm{lift}}^{\infty}(X)$ and $\phi_0\simeq \phi_1$ then  $\phi_1 \in \mathrm{Diff}_{\mathrm{lift}}^{\infty}(X)$. In particular, if $\phi_0\simeq \mathrm{id}_X$ then there exists $\Psi \in \mathrm{Aut}(G)$ such that $\Psi \simeq \mathrm{id}_G$ and $\overline{\Psi}=\phi_0$.
	\item Any two lifts of $\phi_0$ differ by a smooth natural transformation, or, equivalently, there exists a group isomorphism  $\mathrm{Aut}(G)_0/\mathrm{Bis}(G)_0 \simeq \mathrm{Diff}_0^{\infty}(X)$.
	\end{enumerate}
   \end{proposition}

  \begin{proof} 
  	
  Clearly, statement (1) follows directly from Theorem \ref{thm2}. Let us focus on showing statement (2).  Pick $\phi_0\in \mathrm{Diff}_0^{\infty}(X)$ and suppose that $\Psi$ and $\Psi'$ are two lifts for $\phi_0$. We consider their infinitesimal generators $Z = \frac{\partial}{\partial \tau} \oplus Z_\tau$ and $Z' = \frac{\partial}{\partial \tau} \oplus Z'_\tau$, as well as $V = \frac{\partial}{\partial \tau} \oplus V_\tau$, following the notation in the proof of Theorem \ref{thm2}. It follows from Proposition \ref{prop:exact-sequence} that $Z'=\frac{\partial}{\partial \tau}\oplus Z'_\tau=\frac{\partial}{\partial \tau}\oplus( Z_\tau+\delta(b_\tau))=Z+\delta(b)$ where $b:\mathbb{R}\times M\to \mathbb{R}\times A$ is defined by $b(\tau,x)=(\tau,b_\tau(x))$ with $b_\tau\in \Gamma(A)$ for every $\tau\in \mathbb{R}$.
    
 Observe that for every $\tau\in \mathbb{R}$, the vector field $\rho(b_\tau)$ is tangent to the groupoid orbits of $G$. As these orbits are compact, one gets that the maximal integral curves of $\rho(b_\tau)$ are defined for all time. In other words, $\rho(b_\tau)$ is complete, which in turn implies that the vector field $\delta(b_\tau)$ is complete as well. Hence, according to the Trotter product formula for the flow of $Z'=Z+\delta(b)$ (consult \cite[Cor.~4.1.27]{AMR}), one has that:
\begin{align*}
    \varphi_\tau^{Z'}(\varsigma,g)=&\lim_{n\to \infty}\left(\varphi_{\tau/n}^{Z}\circ \varphi_{\tau/n}^{\delta(b)}\right)^n(\varsigma,g)=\lim_{n\to \infty}\left(\varphi_{\tau/n}^{Z}\circ I_{\exp(\tau/n b)}\right)^n(\varsigma,g)\\
    =&\lim_{n\to \infty}(I_{(\prod_{k=1}^{n}\alpha_{\varphi_{k\tau/n}^{Z}}(\exp(\tau/n b)))}\circ \varphi_{\tau}^{Z})(\varsigma,g),
\end{align*}
where we have used the auxiliary relations described in Lemma \ref{product-identities} to obtain the last equality. In consequence, the latter implies that

$$\varphi_\tau^{Z'}(\varsigma,g)=(I_{(\lim_{n\to \infty}\prod_{k=1}^{n}\exp(\mathrm{Lie}(\varphi_{k\tau/n}^{Z})_*(b))\cdot \frac{\tau}{n})}\circ \varphi_{\tau}^{Z})(\varsigma,g).$$
If we rewrite this equation in terms of the time-ordered exponential\footnote{The \textbf{time-ordered exponential} $\mathcal{T}\exp$ of a function $f$ is by definition the expression $$\mathcal{T}\exp(\int_0^\tau f(\theta)d\theta)=\mathcal{T}\exp(\lim_{n\to \infty}\sum_{k=1}^nf(\frac{k}{n}\tau)\frac{\tau}{n}):=\lim_{n\to \infty}\prod_{k=1}^n\exp(f(\frac{k}{n}\tau)\frac{\tau}{n}).$$} $\mathcal{T}\exp$ one gets\footnote{The reader may compare these formulas with those in \cite[Sec.~2.3.2]{BaezSchreiber04} and \cite[Sec.~1.4]{Igusa09}.} that:
\begin{align*}
     \varphi_\tau^{Z'}(\varsigma,g)=(I_{\mathcal{T}\exp(\int_0^\tau\mathrm{Lie}(\mathrm{Exp}(\theta Z))_*(b) d\theta)}\circ \varphi_{\tau}^{Z})(\varsigma,g), \quad  (\varsigma,g)\in \mathbb{R}\times G.
\end{align*}

 Observe that $\mathcal{T}\exp(\int_0^\tau\mathrm{Lie}(\mathrm{Exp}(\theta Z))_*(b) d\theta)$ determines a bisection of $\mathbb{R}\times G$. Since every bisection $\sigma\in \mathrm{Bis}(\mathbb{R}\times G)$ has the form $\sigma:\mathbb{R}\times M\to \mathbb{R}\times G$ with  $\sigma(\tau,x)=(\tau,\sigma_{\tau}(x))$ and $\sigma_{\tau}\in \mathrm{Bis}(G)$ for every $\tau\in \mathbb{R}$, we denote the corresponding time-dependent bisection of $\mathcal{T}\exp(\int_0^\tau\mathrm{Lie}(\mathrm{Exp}(\theta Z))_*(b) d\theta)$ by $\sigma_\tau(Z,b)$. After  considering the corresponding evolution operators we have that
$$\Phi_{\tau+\varsigma,\varsigma}^{Z'}(g)=(I_{\sigma_{\tau+\varsigma}(Z,b)}\circ \Phi_{\tau+\varsigma,\varsigma}^Z)(g).$$ 

\noindent Thus, taking $\varsigma=0$ one sees that $\Phi_{\tau,0}^{Z'}=I_{\sigma_\tau(Z,b)}\circ \Phi_{\tau,0}^Z$. In other words, $[\Phi_{\tau,0}^{Z'}]=[\Phi_{\tau,0}^{Z}]$ in the quotient $\mathrm{Aut}_0(G)/\mathrm{Bis}_0(G)$. Moreover, as consequence of \cite[Prop.~4.1]{HCV23}, we know that there exists a smooth natural transformation $\Phi_{\tau,0}^{Z'}\stackrel{\alpha_\tau(Z,b)}{\Rightarrow}\Phi_{\tau,0}^Z$ for every $\tau\in \mathbb{R}$, which finishes the proof.
\end{proof}

Under the hypotheses of Theorem \ref{thm2}, the following consequence holds:

  \begin{corollary}
Any two lifts for $F$ differ by a smooth natural transformation.
  \end{corollary}
  \begin{proof}
   Indeed, let $\Psi', \Psi$ be two lifts for $F$ and let $Z_\tau', Z_\tau$ denote their infinitesimal generators. Recall that $\Psi_0=\Psi'_0=\psi$. Hence, from Proposition \ref{cor1} it follows that there exists a time-dependent smooth natural transformation $\alpha_\tau(Z,b)$ such that $\Phi_{\tau,0}^{Z'}\stackrel{\alpha_\tau(Z,b)}{\Rightarrow}\Phi_{\tau,0}^Z$. This immediately implies that $\Psi_\tau'\stackrel{\alpha_\tau(Z,b)\circ \psi_0}{\Rightarrow}\Psi_\tau$ for all $\tau\in \mathbb{R}$.
  \end{proof}

We conclude the section by noting that, following the approach of Davis in \cite{Davis}, a proper Lie groupoid with proper orbit projection can be viewed as a collection of generalized principal bundles
$\{G_{(\alpha)}\rightrightarrows M_{(\alpha)}\to X_{\alpha}\}$ that are smoothly glued together into a stratified locally trivial fibration. From this perspective, we may interpret $I_{\sigma}$ as a stratified gauge transformation and $\delta(a)$ as a stratified gauge field, for any $\sigma\in \mathrm{Bis}(\sigma)$ and $a \in \Gamma(A)$.

\end{document}